\newcommand{\powser}[1]{[\![#1]\!]}
\newcommand{\pdiv}{$p$-divisible }
\newcommand{\G}{\mathbb{G}}

\newcommand{\Q}{\mathbb{Q}}
\newcommand{\Z}{\mathbb{Z}}
\newcommand{\QZ}{\Q_p/\Z_p}  
\newcommand{\Zp}[1]{\Z/p^{#1}}

\newcommand{\al}{\alpha}
\newcommand{\Lk}{\Lambda_k}

\newcommand{\lra}[1]{\overset{#1}{\longrightarrow}}

\newcommand{\Prod}[1]{\underset{#1}{\prod}}
\newcommand{\Coprod}[1]{\underset{#1}{\coprod}}
\newcommand{\Colim}[1]{\underset{#1}{\colim}}

\newcommand{\E}{E_{n}}

\def \mmod{/\mkern-3mu /}

\title{An Introduction to HKR Character Theory}
\author{
Nathaniel Stapleton \\   
}
\date{\today}
\documentclass[10pt]{article}
\usepackage{amssymb,amsfonts,amsthm,amsmath,verbatim, mathrsfs}
\usepackage[all]{xy}
\theoremstyle{definition} 
\DeclareMathOperator{\Aut}{Aut}

\DeclareMathOperator{\GL}{GL}
\DeclareMathOperator{\im}{im}
\DeclareMathOperator{\Fix}{Fix}

\DeclareMathOperator{\colim}{colim}

\DeclareMathOperator{\Spec}{Spec}
\DeclareMathOperator{\Spf}{Spf}

\DeclareMathOperator{\sub}{sub}
\DeclareMathOperator{\level}{level}

\newtheorem{thm}[subsubsection]{Theorem}
\newtheorem{example}[subsubsection]{Example}
\newtheorem{prop}[subsubsection]{Proposition}
\newtheorem{cor}[subsubsection]{Corollary}
\newtheorem{lemma}[subsubsection]{Lemma}

\newtheorem{definition}[subsubsection]{Definition}
\newtheorem*{problem}{Problem}
\newtheorem{surprise}[subsubsection]{Surprise}
\newtheorem*{solution}{Solution}
\newtheorem{note}[subsubsection]{Note}
\newtheorem*{remark}{Remark}

\begin{document}
\maketitle

\section{Motivation}

Traditionally, character theory provides a tool for understanding the representations of a finite group. The representations fit together to give the representation ring of the group and the character map is a ring map
\[
R(G) \lra{} Cl(G,\mathbb{C}),
\]
where $R(G)$ is the representation ring and $Cl(G, \mathbb{C})$ is the ring of class functions on $G$ taking values in $\mathbb{C}$. The representation ring shows up in equivariant cohomology as the equivariant $K$-theory of a point:
\[
K_{G}^{0}(\ast) \cong R(G).
\]
Equivariant $K$-theory is quite unique. It is geometrically defined, global equivariant, height $1$, and has an associated commutative group scheme 
\[
\G_m = \Spec K_{S^1}^{0}(\ast).
\]
At heights above $1$ no known analogues exist, although there has been an enormous amount of deep work at height $2$ (\cite{Henriques-Survey},\cite{Lurie-Survey},\cite{Stolz-Teichner-Survery2}).

The completion theorem of Atiyah and Segal \cite{AtiyahSegal} describes the canonical map
\[
K_{G}^{0}(\ast) \lra{} K^{0}(BG)
\]
in terms of the representation ring. The map is just completion at the ideal of representations of virtual dimension zero. The completion map need not be injective or surjective, but it does imply that the $K$-theory of $BG$ ``knows" something about representation theory (and $G$-bundles more generally). 

At a prime $p$ there are excellent higher chromatic analogues of $K$-theory given by height $n$ Morava $E$-theory $\E$. This is supposed to provide motivation for studying $\E^0(BG)$. It should contain a shadow of the inaccessible geometric equivariant cohomology theory that we believe lurks in the background at higher heights.

But how do we study $\E^0(BG)$? One way to go about this is to construct a higher height analogue of the character map: a map that begins with $\E^0(BG)$ and lands in a ring that is reminiscent of class functions on $G$. The character theory of Hopkins, Kuhn, and Ravenel does just this and it provides a useful mechanism with which to study higher height phenomena associated to $\E$. 

\paragraph*{Acknowledgements} This paper grew out of the notes for my talk on HKR character theory at the 2013 Talbot workshop and the paper is written with the Talbot participants in mind as the audience. I would like to thank everyone there for a wonderful week and for their constructive comments. Particularly I would like to thank Mark Behrens and Tyler Lawson for their mentorship, Dylan Wilson for listening to a pre-version of the talk, and Sebastian Thyssen for carefully reading the first version of this paper. This paper would not exist without Eric Peterson's insistence that I write it down and it owes a lot to conversations with David Gepner, Charles Rezk, and Chris Schommer-Pries. 

\begin{remark}
The results of this paper are due to Mike Hopkins, Nick Kuhn, and Doug Ravenel. The majority of the results are proved in Sections 5 and 6 of \cite{hkr}. It will be indicated when the result is from somewhere else. I hope that by filling this introduction with examples, some different points of view, and applications I have provided something that helps those who are trying to gain perspective on this subject. 
\end{remark}

\section{Introduction}
The character theory of Hopkins, Kuhn, and Ravenel is motivated by the following problem:

\begin{problem} \label{problem}
Let $\E$ be Morava E-theory. How can one approximate $\E^*(X)$ by rational cohomology?
\end{problem}

In order to attack this problem we need to understand what we mean by ``approximate" and what we mean by $X$. In this paper we will describe a complete solution to this problem for all spaces of the form $EG\times_G X$, in which $G$ is a finite group and $X$ is a $G$-space with the homotopy type of a finite $G$-CW complex.

Here is an easy case of the problem:

\begin{solution}
When $X$ has the homotopy type of a finite CW-complex the canonical map
\[
\E^*(X) \lra{} (p^{-1}\E)^*(X)
\]
becomes an isomorphism when the domain is tensored up to $(p^{-1}\E)^*$:
\[
(p^{-1}\E)^*\otimes_{\E^*}\E^*(X) \cong (p^{-1}\E)^*(X).
\]
\end{solution}
\begin{proof}
The ring $(p^{-1}\E)^*$ is a flat $\E^*$-algebra (it is also a $\Q$-algebra). Given a cohomology theory defined on finite CW-complexes, flat extension produces another cohomology theory. Note that this is not true for cohomology theories defined on all CW-complexes because infinite wedges turn into infinite products. Now it is completely clear that the isomorphism holds for a point and so we are done. 
\end{proof}

Problems arise immediately when the space $X$ is not finite:

Let $X = B\Z/p^k$. We can mimic the above construction. The map of spectra $\E \lra{} p^{-1}\E$ induces a map
\[
\E^*(B\Z/p^k) \lra{} (p^{-1}\E)^*(B\Z/p^k)
\]
and we can tensor the domain up to $(p^{-1}\E)^*$
\[
(p^{-1}\E)^* \otimes_{\E^*}\E^*(B\Z/p^k) \lra{} (p^{-1}\E)^*(B\Z/p^k).
\]
However, the codomain of this map is trivial (a rank one $(p^{-1}\E)^*$-module) because the rational cohomology of finite groups is trivial (by Maschke's theorem). The domain is a free $(p^{-1}\E)^*$-module of rank $p^{kn}$ because $\E^0(B\Z/p^k)$ is the global sections of the $p^k$-torsion of the formal group associated to $\E$. Thus we must modify the space that shows up in the codomain in order for this to have any chance of being an isomorphism. 

What is the real issue here? In order to get into it we must think about some algebraic geometry. Whenever possible, it is useful to think of the Morava $E$-theory of a space (and more generally the complex orientable cohomology of a space) algebro-geometrically. Although this may seem tautological because the schemes involved are all affine, real progress has been made regarding our understanding of these cohomology theories by taking this point of view seriously.

\section{Algebraic Geometry}
\subsection{Preliminaries}
Recall that 
\[
\E^0 \cong W(k)\powser{u_1,\ldots,u_{n-1}}
\]
and that $\E$ has an associated formal group 
\[
\G_{\E} = \Spf \E^0(BS^1).
\]
This is the universal deformation of a height $n$ formal group over the field $k$. A good reference for the construction of $\E$ and its basic properties (as well as a proof that it is an $A_{\infty}$-ring spectrum) is \cite{Rezk-Notes}. For doing character theory it is very useful to view this as a \pdiv group. 

\begin{definition}
A \pdiv group over a commutative ring $R$ of height $n$ is an inductive system $(G_v,i_v) = G_1 \lra{i_1} G_2 \lra{i_2} \ldots$ such that
\begin{enumerate}
\item $G_v$ is a finite flat commutative group scheme over $R$ of order $p^{vn}$;
\item For each $v$, there is an exact sequence 
\[
0 \lra{} G_v \lra{i_v} G_{v+1} \lra{p^v} G_{v+1},
\]
where $i_v$ is the natural inclusion and $p^v$ is multiplication by $p^v$ in $G_{v+1}$.
\end{enumerate}
\end{definition}
A finite flat group scheme of order $p^{vn}$ is the scheme associated to a Hopf algebra that is finitely generated and flat (free for the purposes of this paper) of (local) rank $p^{vn}$ as a module over the base.

The \pdiv group associated to $\E$ is the system of finite flat group schemes
\[
\G_{\E}[p] \lra{} \G_{\E}[p^2] \lra{} \ldots,
\]
where $\G_{\E}[p^k]$ is the $p^k$-torsion in the formal group $\G_{\E}$. This \pdiv group actually contains equivalent information to the formal group by Tate's Lemma 0 from \cite{Tate-p-div}. 

It follows from Lemma 5.7 of \cite{hkr} that 
\[
\G_{\E}[p^k] \cong \Spec(\E^0(B\Z/p^k))
\]
and that with a choice of coordinate
\[
\E^0(B\Z/p^k) \cong \E^0\powser{x}/[p^k](x) \cong \E^0[x]/f(x),
\]
where $[p^k](x)$ is the $p^k$-series for the formal group law associated to the coordinate and $f(x)$ is a monic degree $p^{kn}$ polynomial. The $n$ in the exponent is the height of the \pdiv group. The last isomorphism is an application of the Weierstrass preparation theorem. 

It should be noted that we can take $\Spec(-)$ and not $\Spf(-)$ here because $\E^0(B\Z/p^k)$ is a free $\E^0$-module of rank $p^{kn}$. Because it is a finitely generated $\E^0$-algebra it is well behaved (unlike $\E^0(BS^1) \cong \E^0\powser{x}$).

There is a fully faithful embedding of the category of finite groups into the category of finite flat group schemes over $\Spec(R)$ by sending $G$ to
\[
\Spec(\Prod{G}R).
\]
We call finite flat group schemes of this form constant group schemes. 
The only constant height $n$ \pdiv group is
\[
(\Q_p/\Z_p)^n = (\Z/p^{\infty})^n = ((\Z/p)^n \lra{} (\Z/p^2)^n \lra{} \ldots).
\]
From now on we will write $\QZ^n$ for $(\QZ)^n$. For more information about \pdiv groups there are the classic books \cite{Dem} and \cite{messing} and the articles \cite{Shatz} and \cite{Tate-p-div}. For more information about finite flat group schemes there is Tate's article \cite{Tate-finiteflat}. 

Lastly, we need to understand how to base change \pdiv groups. We do the most obvious thing. For an $\E^0$-algebra $R$, let
\[
R\otimes \G_{\E} := (R\otimes \G_{\E}[p] \lra{} R\otimes \G_{\E}[p^2] \lra{} \ldots),
\]
where
\[
R \otimes \G_{\E}[p^k] := \Spec(R) \times_{\Spec(\E^0)} \G_{\E}[p^k].
\]
Note that for a constant group scheme $G$ over $\Spec(S)$ and an $S$-algebra $R$, $R \otimes G$ is the constant group scheme $G$ over $\Spec(R)$.

\subsection{Construction of $C_0$}
Before the interlude on algebraic geometry we decided that we were interested in the object
\[
p^{-1}\E^0 \otimes_{\E^0} \E^0(B\Z/p^k).
\]
We have switched to working with the zeroth coefficients for convenience, everything does hold in the graded setting. From the discussion in the previous section we recognize this ring as the global sections of
\[
p^{-1}\E^0 \otimes \G_{\E}[p^k].
\]
It is a result of Cartier's discussed in \cite{Shatz} that this is an \'etale finite flat group scheme (because $p^{-1}\E^0$ is a $\Q$-algebra). In fact he shows that any finite flat group scheme is \'etale after inverting $p$. Instead of working at the $p^k$-torsion we could consider the \pdiv group 
\[
\G_{et} := p^{-1}\E^0 \otimes \G_{\E}.
\]
This is called an \'etale \pdiv group because the $p^k$ torsion is \'etale for every $k$.
\'Etale finite flat group schemes are not too far away from being constant group schemes. In fact we can easily describe an extension of $p^{-1}\E^0$ over which $\G_{et}[p^k]$ is constant. We could just take the algebraic closure of the fraction field of $p^{-1}\E^0$. So let's assume for the moment that we have a $p^{-1}\E^0$-algebra $C_0$ over which 
\[
C_0 \otimes \G_{et} \cong \QZ^n
\]
so 
\[
C_0 \otimes \G_{et}[p^k] \cong (\Zp{k})^n.
\]
On the level of cohomology rings this means that
\[
C_0 \otimes_{\E^0} \E^0(B\Zp{k}) \cong \Prod{(\Zp{k})^n}C_0 \cong C_{0}\otimes_{p^{-1}\E^0}(p^{-1}\E)^0(\Coprod{(\Zp{k})^n} \ast).
\]
We can rewrite the right hand side as
\[
C_{0}^{0}(\Coprod{(\Zp{k})^n} \ast).
\]
Thus over $C_0$ we can solve Problem \ref{problem} for an infinite space.

\begin{surprise}
Over the ring $C_0$ we can solve Problem \ref{problem} for all spaces of the form $EG \times_G X$.
\end{surprise}

But what is $C_0$? Instead of taking an extension of $p^{-1}\E^0$ that suffices, it would be much more satisfying to write down the universal extension of $p^{-1}\E^0$ equipped with an isomorphism
\[
\QZ^{n} \lra{\cong} C_0 \otimes \G_{et}.
\]
This is what we do now.

Let $A$ be a finite abelian group such that $A = A[p^k]$. Let
\[
A^* = \hom(A, S^1) \cong \hom(A, \Z/p^k)
\]
be the Pontryagin dual of $A$. The isomorphism can be made canonical by setting $\Z/p^k = S^1[p^k]$.

With this setup, we will use the following key lemma of \cite{hkr}:
\begin{lemma} \label{isolemma}
Let $R$ be an $\E^0$-algebra. There is a natural isomorphism of functors
\[
\hom_{\E^0\text{-alg}}(\E^0(BA),R) \cong \hom_{\text{gp-scheme}}(R\otimes A^*, R\otimes \G_{\E}[p^k]).
\]
\end{lemma}
\begin{proof}
(Sketch) An $\E^0$-algebra map
\[
f:\E^0(BA) \lra{} R
\]
is equivalent to the $R$-algebra map
\[
f:R\otimes_{\E^0}\E^0(BA) \lra{} R.
\]
Now $a \in A^*$ gives a map
\[
R\otimes_{\E^0} \E^0(B\Z/p^k) \lra{} R \otimes_{\E^0} \E^0(BA) \lra{f} R.
\]
Using all of the elements of $A^*$ gives a map
\[
R\otimes_{\E^0} \E^0(B\Z/p^k) \lra{} \Prod{A^*}R
\]
in the codomain.
\end{proof}

\begin{cor}
For $R$ a $p^{-1}\E^0$-algebra, there is a natural isomorphism
\[
\hom_{p^{-1}\E^0\text{-alg}}(p^{-1}\E^0\otimes_{\E^0}\E^0(BA),R) \cong \hom_{\text{gp-scheme}}(R\otimes A^*, R\otimes \G_{et}[p^k]).
\]
\end{cor}
\begin{proof}
This follows by pulling the isomorphism in the previous lemma back to $p^{-1}\E^0$.
\end{proof}
If we take $A = (\Zp{k})^n =: \Lk$ and set
\[
R = C_{0,k}' := p^{-1}\E^0\otimes_{\E^0}\E^0(B\Lk)
\]
then the identity map on the left side provides a canonical map
\[
\Lk^* \lra{} C_{0,k}' \otimes \G_{et}[p^k].
\]
Of course, $\Lk^* \cong \QZ^n[p^k]$ is non-canonically isomorphic to $\Lk$. Algebraically we can think of this as formally adjoining roots to the polynomial $f(x)$ described in the beginning of this section. To force the map to be an isomorphism we need to invert the determinant of the map on global sections.

Evaluating the above map at $C_{0,k}'$ gives a map (of sets)
\[
\Lk^* \lra{z} \G_{et}[p^k](C_{0,k}').
\]
This picks out the ``formal roots" of $f(x)$.

\begin{lemma} \label{det}
Inverting the determinant is equivalent to inverting the nonzero image of $z$.
\end{lemma}
\begin{proof}
Fixing a coordinate $x$ gives an isomorphism
\[
C_{0,k}' \cong p^{-1}\E^0 \otimes_{\E^0} \E^0 \powser{x_1, \ldots, x_n}/([p^k](x_1),\ldots, [p^k](x_n)).
\]
For $\bar{i} = (i_1, \ldots, i_n) \in \Lk^*$, let
\[
[\bar{i}](\bar{x}) = [i_1](x_1) +_{\G_{et}} \ldots +_{\G_{et}} [i_n](x_n).
\]
Then $\bar{i} \mapsto [\bar{i}](\bar{x}) \in \G_{et}[p^k](C_{0,k}')$. The global sections of
\[
\Lk^* \lra{} C_{0,k}' \otimes \G_{et}[p^k]
\]
is given by the Vandermonde matrix
\[
\left(
\begin{array}{ccccc}
1 & [\bar{j_1}](\bar{x}) & [\bar{j_1}](\bar{x})^2 & \ldots & [\bar{j_{1}}](\bar{x})^{p^k-1} \\
1 & [\bar{j_2}](\bar{x}) & [\bar{j_2}](\bar{x})^2 & \ldots & [\bar{j_{2}}](\bar{x})^{p^k-1}\\
\vdots & \vdots & \vdots & \ddots & \vdots \\
1 & [\bar{j_{p^k}}](\bar{x}) & [\bar{j_{p^k}}](\bar{x})^2 & \ldots & [\bar{j_{p^k}}](\bar{x})^{p^k-1}\\
\end{array} \right),
\]
where $\{\bar{j_1},\ldots,\bar{j_{p^k}}\} = \Lk^*$.

Now the determinant of the matrix is
\[
\Prod{1 \leq s < t \leq p^k} ([\bar{j_t}](\bar{x}) - [\bar{j_s}](\bar{x})).
\]
It is a basic fact of subtraction for formal group laws that
\[
x -_{\G} y = u*(x-y)
\]
for a unit $u$. 

Thus the determinant is, up to multiplication by a unit, 
\begin{eqnarray*}
\Prod{1 \leq s < t \leq p^k} ([\bar{j_t}](\bar{x}) -_{\G_{et}} [\bar{j_s}](\bar{x})) & = & \Prod{1 \leq s < t \leq p^k} ([\bar{j_t} - \bar{j_s}](\bar{x})) \\
& = & \Prod{0 \neq \bar{j} \in \Lk^*} [\bar{j}](\bar{x})^{p^k-1}.
\end{eqnarray*} 
\end{proof}

Let 
\[
C_{0,k} := (det)^{-1}C_{0,k}'.
\]
The image of $z$ consists largely of zero divisors. Because of this, we should be concerned that $C_{0,k}$ is the zero ring. Proposition 6.5 of \cite{hkr} saves the day:

\begin{lemma}
The ring $C_{0,k}$ is a faithfully flat $p^{-1}\E^0$-algebra.
\end{lemma}

Over the ring
\[
C_0 = \Colim{k} C_{0,k}
\]
there is a canonical isomorphism of \pdiv groups
\[
\QZ^{n} \lra{\cong} C_0 \otimes \G_{et}.
\]

This is because $C_0$ is a $C_{0,k}$-algebra for all $k$. Thus the canonical map $p^{-1}\E^0\otimes_{\E^0}\E^0(B\Lk) \lra{} C_0$ corresponds to an isomorphism
\[
\Lk^* \lra{\cong} C_0 \otimes \G_{et}[p^k].
\]
But more, the maps fit together for all $k$ giving a map 
\[
\Colim{k} \text{ } p^{-1}\E^0\otimes_{\E^0}\E^0(B\Lk) \lra{} C_0
\]
and this implies the isomorphism.
\begin{remark}
It may be helpful to look at Proposition 2.17 in \cite{tgcm}. It provides a description of $C_0$ as the $p^{-1}\E^0$-algebra that represents the functor of isomorphisms from $\QZ^n$ to $\G_{et}$. 
\end{remark}

We can use $C_0$ to construct a new cohomology theory by extension of coefficients:
\[
C_{0}^0(X) = C_0 \otimes_{p^{-1}\E^0} (p^{-1}\E)^0(X).
\]

\begin{example} \label{padicktheory}
Let us work out $C_0$ in the case of $E_1 = K_p$. Let 
\[
\G_{K_p} = \hat{\G}_m
\]
be the formal multiplicative group. We can choose a coordinate so that the group law is
\[
x+_{\G_m}y = x+y+xy.
\] 
In this case something nice happens, we have that
\[
\G_m[p^k] \cong \hat{\G}_m[p^k].
\]
That is, the $p^k$-torsion of the multiplicative formal group is the $p^k$-torsion of a global geometric object: the multiplicative group.

On global sections and with the standard coordinate this is the map
\[
\Z_p[x]/([p^k](x)) \cong \Z_p[x]/(x^{p^k}-1)
\]
given by
\[
x \mapsto x-1.
\]
Thus we have that 
\[
C_{0,k}' = \Q_p[x]/([p^k](x)) \cong \Q_p[x]/(x^{p^k}-1).
\]
The non-zero image of $\Lk^* \cong \Zp{k}$ under $z$ is just 
\[
\{x,[2](x),\ldots, [p^k-1](x)\},
\]
which is
\[
\{x-1,x^2-1, \ldots, x^{p^k-1}-1\}
\]
under the isomorphism.
There is the factorization
\[
x^{p^k}-1 = \Prod{i=1,\ldots,k}\zeta_{p^i}(x),
\]
where $\zeta_{p^i}(x)$ is the $p^i$th cyclotomic polynomial. But
\[
x^{p^{k-1}}-1 = \Prod{i=1,\ldots,(k-1)}\zeta_{p^i}(x)
\]
and this is one of the elements in the image of $z$. Clearly it is a zero-divisor in $C_{0,k}'$. Thus the canonical map
\[
\Q_p[x]/(x^{p^k}-1) \lra{} (x^{p^{k-1}}-1)^{-1}\Q_p[x]/(x^{p^k}-1)
\]
factors through the quotient map
\[
\Q_p[x]/(x^{p^k}-1) \lra{} \Q_p[x]/(\zeta_{p^k}(x)).
\]
But $\Q_p[x]/(\zeta_{p^k}(x))$ is a field so the canonical map
\[
\Q_p[x]/(\zeta_{p^k}(x)) \lra{\cong} (x^{p^{k-1}}-1)^{-1}\Q_p[x]/(x^{p^k}-1)
\]
is an isomorphism. Thus $C_{0,k}$ is just $\Q_p$ adjoin a primitive $p^k$th root of unity and
\[
C_0 = \Colim{k}\text{ }C_{0,k} = \Colim{k}\text{ }\Q_p(\zeta_{p^k})
\]
is $\Q_p$ adjoin all prime power roots of unity - the maximal ramified extension of $\Q_p$. \qed
\end{example}

\section{The Character Map}
\subsection{The construction}
For $X$ a finite $G$-CW complex, the character map takes the form
\[
\Phi_G: \E^0(EG \times_G X) \lra{} C_{0}^0(EG \times_G \Fix_n(X)),
\]
where $\Fix_n(X)$ is a finite $G$-CW complex that we define below. The map is the composite of two maps. The first is induced by a map of topological spaces and the second is defined algebraically using the definition of $C_0$. To begin, we will define $\Fix_n(X)$ and give some examples.

\begin{definition}
For a finite group $G$, let
\[
G_{p}^{n} = \hom(\Z_{p}^{n},G).
\]
\end{definition}

This is just $n$-tuples of commuting prime-power order elements of $G$:
\[
G_{p}^{n} \cong \{(g_1,\ldots, g_n)|[g_i,g_j] = e, \text{ } g_{i}^{p^k}=e \text{ for } k \gg 0\}.
\]
The group $G$ acts on this set by conjugation and we will write $G_{p}^{n}/\sim$ for the quotient by this action.

\begin{definition} \label{fix}
For a finite $G$-CW complex $X$, let
\[
\Fix_n(X) = \Coprod{\al \in G_{p}^{n}} X^{\im \al},
\]
where $X^{\im \al}$ is the fixed points of $X$ with respect to the image of $\al$.
\end{definition}

This is a finite $G$-CW complex by intertwining the action of $G$ on $G_{p}^{n}$ by conjugation and the action of $G$ on $X$. To be precise, for $x \in X^{\im \al}$, we let $gx \in X^{\im g \al g^{-1}}$.

The $G$-space $\Fix_n(X)$ is also known as the $n$-fold inertia groupoid of the $G$-space $X$ (really a $p$-complete form of the inertia groupoid). That is, the definition above can be made for all finite groups uniformly in the category of topological groupoids. Let $X \mmod G$ be the action topological groupoid associated to the finite $G$-CW complex $X$. This looks like
\[
X \times G \rightrightarrows X,
\]
where the maps are the source and target maps (of course there are also composition, inverse, and identity maps). 

The internal mapping topological groupoid between two action groupoids $W \mmod F$ and $X\mmod G$, $\text{Hom}_{top.gpd}(W \mmod F, X \mmod G)$, can be described in the following way. The space of functors between the two groupoids is given by the fiber product
\[
\xymatrix{\hom_{top.gpd}(W \mmod F, X \mmod G) \ar[r] \ar[d] & \hom_{top}(W \times F, X \times G) \ar[d]^{(s_*,t_*)} \\ \hom_{top}(W,X) \ar[r]^-{(s^*,t^*)} & \hom_{top}(W \times F, X) \times \hom_{top}(W \times F, X).}
\]

Let $[\underline{1}] := (0 \rightarrow 1)$ be the free-standing isomorphism. Now we have that
\[
\text{Hom}_{top.gpd}(W \mmod F, X \mmod G) = \big(\hom_{top.gpd}(W \mmod F \times [\underline{1}], X \mmod G) \rightrightarrows \hom_{top.gpd}(W \mmod F, X \mmod G)\big),
\]
where the maps are induced by the source and target maps in $[\underline{1}]$.

\begin{lemma}
Let $\ast \mmod \Z_{p}^{n}$ be the topological groupoid with a single object with $\Z_{p}^{n}$ automorphisms. There is a natural isomorphism 
\[
\text{Hom}_{top. gpds}(\ast \mmod \Z_{p}^{n}, X \mmod G) \cong \Fix_n(X) \mmod G.
\]
\end{lemma} 
\begin{proof}
We indicate the proof. We only need to do this for $n=1$ as higher $n$ follow by adjunction. It also suffices to replace $\ast \mmod \Z_p$ by $\ast \mmod \Z/p^k$ for $k$ large enough. Now a functor
\[
\ast \mmod \Z/p^k \lra{} X \mmod G
\]
picks out an element of $X$ and a prime power element of $G$ that fixes it. Thus the collection of functors are in bijective correspondence with
\[
\Coprod{\al \in G^{1}_{p}} X^{\im \al}.
\]
A natural transformation between two functors is a map
\[
\xymatrix{\Z/p^k \coprod \Z/p^k \coprod \Z/p^k \ar[r] \ar@<1ex>[d] \ar@<-1ex>[d] & G\times X \ar@<1ex>[d] \ar@<-1ex>[d] \\ \ast \coprod \ast \ar[r] & X.}
\]
The domain here is $\ast \mmod \Z/p^k \times [\underline{1}]$. The two $\Z/p^k$'s on the sides come from the identity morphisms in $[\underline{1}]$. The $\Z/p^k$ in the middle is the important one. Given two functors picking out $x_1 \in X^{\im \al_1}$ and $x_2 \in X^{\im \al_2}$, a commuting diagram above implies that a morphism from $x_1$ to $x_2$ in the inertia groupoid is an element $g \in G$ sending $x_1$ to $x_2$ in $X$. However, the composition diagram implies that $g$ must conjugate $\al_1$ to $\al_2$.
\end{proof}

We now compute $EG \times_G \Fix_n(X)$ in several cases.

\begin{example} \label{xapoint}
When $X = \ast$,
\[
EG \times_G \Fix_n(\ast) \cong EG \times_{G}^{\text{conj}} G_{p}^{n}.
\]
Fixing an element $\al$ in a conjugacy class $[\al]$ of $G_{p}^{n}$, the stabilizer of the element is precisely the centralizer of the image of $\al$ in $G$. Thus there is an equivalence
\[
\pushQED{\qed}
EG \times_{G}^{\text{conj}} G_{p}^{n} \simeq \Coprod{[\al] \in G_{p}^{n}/\sim} BC(\im \al). \qedhere
\popQED
\]
\end{example}

\begin{example}
When $G$ is a $p$-group there is an isomorphism
\[
\hom(\Z_{p}^{n},G) \cong \hom(\Z^n, G).
\]
Then there is an equivalence
\[
EG \times_{G}^{conj} G_{p}^{n} \simeq L^nBG,
\]
where $L$ is the free loop space functor. 
\qed
\end{example}

\begin{example} \label{zpk}
When $G = \Zp{k}$,
\[
\pushQED{\qed}
EG \times_{G}^{\text{conj}} G_{p}^{n} \simeq \Coprod{\al \in \Lk^*} B \Zp{k}. \qedhere
\popQED
\]
\end{example}

The topological part of the character map is a map
\[
\E^0(EG \times_G X) \lra{} \E^0(B\Lk \times EG\times_G \Fix_n(X)).
\]

Using the description of $EG \times_G \Fix_n(X)$ in terms of topological groupoids this map is extremely easy to describe. It is induced by the evaluation map
\[
ev: * \mmod \Lk \times \text{Hom}_{top.gpd}(* \mmod \Lk, X \mmod G) \lra{} X \mmod G.
\]

\begin{example} \label{topexample}
When $X = \ast$ this has a very simple description. To illustrate it, recall that a group is abelian if and only if the multiplication map is a group homomorphism. Given an abelian subgroup $A$ of a group $G$ the multiplication map
\[
\times: A \times A \lra{} G
\]
extends on one side to the centralizer of $A$ in $G$. This is a group homomorphism
\[
\times: A \times C(A) \lra{} G.
\]
This is just because
\[
a_1a_2c_1c_2 = a_1c_1a_2c_2
\]
for $a_i \in A$ and $c_i \in C(A)$.
When $X=\ast$, the evaluation map above takes the form
\[
B\Lk \times \Coprod{[\al]} BC(\im \al) \cong \Coprod{[\al]}B(\Lk \times C(\im \al)) \lra{} BG.
\]
The map on the component corresponding $[\al] \in G_{p}^{n} / \mkern-3mu \sim$ is induced by the extension of addition to the centralizer of the image of $\al$:
\[
\pushQED{\qed}
(l,c) \mapsto \al(l)c. \qedhere
\popQED
\]
\end{example}

Because $\E^0(B\Lk)$ is finitely generated and free over $\E^0$, there is a K\"unneth isomorphism
\[
\E^0(B\Lk \times EG\times_G \Fix_n(X)) \cong \E^0(B\Lk) \otimes_{\E^0} \E^0(EG \times_G \Fix_n(X)).
\]

Recall the definition of $C_0$-cohomology. The algebraic part of the character map is a map
\[
\E^0(B\Lk) \otimes_{\E^0} \E^0(EG \times_G \Fix_n(X)) \lra{} C_{0}^{0}(EG \times_G \Fix_n(X))
\]
that is defined as a tensor product of maps.

It is defined on the left bit of the tensor product by the canonical map
\[
\E^0(B\Lk) \lra{} C_0
\]
coming from the canonical map 
\[
\E^0(B\Lk) \lra{} C_{0,k}.
\]
It is defined on the right and bottom bits by the canonical map
\[
\E \lra{} p^{-1}\E.
\]
The topological and algebraic maps fit together to give
\[
\Phi_G: \E^0(EG \times_G X) \lra{} C_{0}^0\otimes_{p^{-1}\E^0}(p^{-1}\E)^0(EG \times_G \Fix_n(X)).
\]

We next discuss several special cases of $\Phi_G$.
\begin{example} \label{classfncs}
When $X = *$, Example \ref{xapoint} gives an isomorphism
\begin{align*}
C_{0}^0(EG \times_G \Fix_n(*)) &\cong C_{0}^{0}(\Coprod{[\al] \in G_{p}^{n}/\sim} BC(\im \al)) \\
&\cong \Prod{[\al] \in G_{p}^{n}/\sim}C_{0}.
\end{align*}
The last isomorphism follows from the fact that $C_{0}^{0}(BG) = C_0$ (because $C_0$ is a $\Q$-algebra). The codomain can be interpreted as class functions on $G_{p}^{n}$ taking values in $C_0$:
\[
\pushQED{\qed}
Cl(G_{p}^{n}, C_0).\qedhere
\popQED
\]
\end{example}

\begin{example} \label{trivG}
When $X$ is a trivial $G$-space, the character map takes the form
\[
\E^0(BG \times X) \lra{} \Prod{[\al] \in G_{p}^{n}/\sim} C_{0}^{0}(BC(\im \al) \times X) \cong \Prod{[\al] \in G_{p}^{n}/\sim} C_{0}^{0}(X).
\]
This is a generalization of the previous example that is often useful. \qed
\end{example}

More generally, because $C_0$ is rational, the $G$ action on $\Fix_n(X)$ can be ``pulled out":
\[
C_{0}^0(EG \times_G \Fix_n(X)) \cong (C_{0}^0(\Fix_n(X)))^G.
\]

In this example we build on the previous example and Example \ref{topexample}.
\begin{example} \label{charmapexample}
When $X = \ast$, the classifying spaces of centralizers $BC(\im \al)$ seem to play a critical role in the definition of the topological part of the character map. However, the algebraic part kills them because $C_0$ is a rational algebra. Thus when $X = \ast$, there is a very simple description of $\Phi_G$: Given $[\al] \in G_{p}^{n}/\sim$ the character map on to the factor corresponding to $[\al]$ is just
\[
\pushQED{\qed}
\E^0(BG) \lra{\E^0(B\al)} \E^0(B\Lk) \lra{\text{can}} C_0. \qedhere
\popQED
\]
\end{example}

The following example explains how $\Phi_G$ was built to reduce to the algebraic geometry of \pdiv groups. 
\begin{example} \label{cyclic}
When $G = \Zp{k}$ and $X = \ast$, we see from Example \ref{zpk} that the character map reduces to
\[
\E^0(B\Zp{k}) \lra{} \Prod{\Lk^*}C_{0}^{0}.
\]
From Example \ref{charmapexample} that the map is induced by the elements of the Pontyagin dual of $\Lk$ in exactly the same way as the construction of the isomorphism in Lemma \ref{isolemma}. Thus this map is precisely the global sections of the canonical map
\[
\pushQED{\qed}
\QZ^n[p^k] \lra{} \G_{\E}[p^k].\qedhere
\popQED
\]
\end{example}

\begin{example}
When $n=1$ and $X = *$ the character map produces a completed version of the classical character map from representation theory that is due to Adams (Section 2 of \cite{Adams-Maps2}). The map takes the form
\[
K_{p}^{0}(BG) \lra{} Cl(G_p, C_0),
\]
where $C_0$ is the maximal ramified extension of $\Q_p$ discussed in Example \ref{padicktheory}. \qed
\end{example}

\begin{note}
When working with a fixed group $G$, it suffices to take $C_{0,k}$ as the coefficients of the codomain of the character map, where $k \geq 0$ is large enough that any continuous map $\Z_p \lra{} G$ factors through $\Zp{k}$.
\end{note}

\subsection{The isomorphism}
It is not hard to prove that the character map is an isomorphism after tensoring the domain up to $C_0$:
\[
C_0\otimes_{\E^0}\Phi_G: C_0\otimes_{\E^0}\E^0(EG \times_G X) \lra{\cong} C_{0}^0(EG \times_G \Fix_n(X)).
\]
The proof is by reduction to the case of $X =\ast$ and $G = \Zp{k}$. Given a finite $G$-CW complex $X$ there is a trick of Quillen's \cite{Quillen-equivariant}, often called ``complex oriented descent", that allows you to reduce to the case of finite $G$-CW complexes with abelian stabilizers. Then Mayer-Vietoris allows you to reduce to the case of spaces of the form $D^n \times G/A$, where $A$ is abelian. Homotopy invariance reduces this to $G/A$ and the Borel construction gives
\[
EG \times_G G/A \simeq  BA.
\] 
As $A$ is finite, the K\"unneth isomorphism brings us back to the algebraic geometry at the beginning of the paper and now Example \ref{cyclic} finishes off the proof.

Most of these steps are pretty clear. We provide an exposition of complex oriented descent. Given a vector bundle over $X$ we can construct the bundle of complete flags
\[
\mathscr{F} \lra{} X. 
\]

Let $k$ be large enough that there is an injection $G \hookrightarrow U(k)$. Then the complete flags on the trivial bundle
\[
X \times \mathbb{C}^k \lra{} X
\]
is just 
\[
X \times U(k)/T \lra{} X,
\]
where $T$ is a maximal torus in $U(k)$.

\begin{lemma} \label{free} (\cite{hkr}, Proposition 2.6) \label{flags}
The ring $\E^0(\mathscr{F})$ is a finitely generated free module over $\E^0(X)$.
\end{lemma}
Hopkins, Kuhn, and Ravenel in fact give an explicit description of $\E^0(\mathscr{F})$ in terms of the Chern classes of the vector bundle, but Lemma \ref{free} is all that we need for complex oriented descent.

The space $U(k)/T$ has the homotopy type of a finite $G$-CW complex and has the property that it only has abelian stabilizers. Thus $X \times U(k)/T$ also has this property. Lemma \ref{flags} implies that
\[
\E^0(X) \lra{} \E^0(X \times U(k)/T) 
\] 
is faithfully flat and the pullback square
\[
\xymatrix{X \times U(k)/T \times U(k)/T \ar[r] \ar[d] & X \times U(k)/T \ar[d] \\
			X \times U(k)/T \ar[r] & X}
\]
implies that
\[
\E^0(X \times U(k)/T \times U(k)/T) \cong \E^0(X \times U(k)/T) \otimes_{\E^0(X)} \E^0(X \times U(k)/T).
\]
Thus the cosimplicial complex
\[
\E^0(X) \rightarrow \E^0(X \times U(k)/T) \rightrightarrows \E^0(X \times U(k)/T \times U(k)/T) \ldots
\]
is just the Amistur complex of the faithfully flat extension
\[
\E^0(X) \rightarrow \E^0(X \times U(k)/T)
\]
and is therefore exact. This implies that $\E^0(X)$ is the equalizer of
\[
\E^0(X \times U(k)/T) \rightrightarrows \E^0(X \times U(k)/T \times U(k)/T)
\]
and thus is determined by the cohomology of finite $G$-spaces with abelian stabilizers.

\subsection{The action of $GL_n(\Z_p)$}
The building blocks of the ring $C_0$ are the rings $\E^0(B\Lk)$. There is an obvious action of $\Aut(\Lk)$ on $\E^0(B\Lk)$. This action pulls back to an action on 
\[
C_{0,k}' = p^{-1}\E^0 \otimes_{\E^0} \E^0(B\Lk).
\]
An automorphism of $\Lk$ sends the image of $z$ (defined right above Lemma \ref{det}) to the image of $z$ so there is an induced action on $C_{0,k}$. Finally, on the colimit $C_0$ this turns into an action of 
\[
\Aut(\QZ^n) \cong GL_n(\Z_p).
\]

\begin{lemma}
The fixed points for the action of $GL_n(\Z_p)$ on $C_0$ are just $p^{-1}\E^0$.
\end{lemma}

\begin{example}
When $n=1$, Example \ref{padicktheory} gives
\[
C_0 = \Colim{k}\text{ }C_{0,k} = \Colim{k}\text{ }\Q_p(\zeta_{p^k}).
\]
This is the maximal ramified extension of $\Q_p$, thus 
\[
\pushQED{\qed}
Gal(C_0/\Q_p) \cong \Z_{p}^{\times} \cong GL_1(\Z_p).\qedhere
\popQED
\]
\end{example}

There is also an action of $GL_n(\Z_p)$ on $\Fix_n(X)$. This is most easily seen by recalling the topological groupoid definition: $\hom_{top. gpds}(\ast \mmod \Z_{p}^{n}, X \mmod G)$. There is an action of $GL_n(\Z_p)$ on the source and this puts an action on the mapping groupoid.

To see the action on $\Fix_n(X)$ defined as in Definition \ref{fix}, let $\sigma \in GL_n(\Z_p)$, then for $x \in X^{\im \al}$ we set
\[
\sigma x = x \in X^{\im (\al \circ \sigma)}.
\]
Here we are using the fact the $\im \al \circ \sigma = \im \al$.

We bring this action up because it turns out that the isomorphism
\[
C_0\otimes_{\E^0}\Phi_G: C_0\otimes_{\E^0}\E^0(EG \times_G X) \lra{\cong} C_{0}^0(EG \times_G \Fix_n(X))
\]
is $GL_n(\Z_p)$-equivariant, where the action is only on the ring $C_0$ in the domain and on both $C_0$ and $\Fix_n(X)$ in the codomain. Taking fixed points gives Theorem C in Hopkins-Kuhn-Ravenel ($\cite{hkr}$) and the solution to Problem \ref{problem} that we have been working towards:

\begin{thm} \label{main}
There is an isomorphism of Borel equivariant cohomology theories
\[
p^{-1}\E^0 \otimes_{\E^0} \Phi_G: p^{-1}\E^0\otimes_{\E^0}\E^0(EG \times_G X) \lra{\cong} (p^{-1}\E)^0(EG \times_G \Fix_n(X))^{GL_n(\Z_p)}.
\]
\end{thm}

\begin{note}
The actions of $G$ and $GL_n(\Z_p)$ commute. We could pull the fixed points for $GL_n(\Z_p)$ inside as a homotopy orbits and then the codomain becomes 
\[
(p^{-1}\E)^0(E(G \times GL_n(\Z_p)) \times_{(G \times GL_n(\Z_p))} \Fix_n(X)).
\]
\end{note}

\subsection{The Drinfeld ring $D$} \label{D}
Instead of considering the $\E^0$-algebra $C_0$ with the property that
\[
C_0 \otimes \G_{\E} \cong \QZ^n,
\]
there is a sort of halfway house in which the prime $p$ has not been inverted. This is the $\E^0$-algebra $D$ over which there is a canonical full level structure (you can think about this as an injection)
\[
\QZ^n \hookrightarrow \G_{\E}.
\]
This ring has a very simple description in terms of $C_0$. Recall that $C_0$ receives a canonical map from $\Colim{k} \text{ } \E^0(B\Lk)$. We define $D$ to be the image of this map in $C_0$. 
This ring is an integral domain by Theorem 7.3 in \cite{subgroups} (also see Theorem 2.4.3 in \cite{Isogenies}), it is $p$-complete and local, and it has the property (from \cite{hkr}) that
\[
p^{-1}D = C_0.
\]

\begin{example}
When $n=1$ the ring $D$ has a very concrete description. Recall that
\[
[p^k](x) = [p][p^{k-1}](x).
\]
Thus we can define
\[
\langle p^k \rangle (x) = [p^k](x)/[p^{k-1}(x)].
\]
Now we have that
\[
\pushQED{\qed}
D = \Colim{k} \text{ } K_{p}^{0}[x]/(\langle p^k \rangle (x)) \cong \Colim{k} \text{ } K_{p}^{0}[x]/\zeta_{p^k}(x). \qedhere
\popQED
\]
\end{example}

There is an action of $\GL_{n}(\Z_p)$ on $D$ just as there is one on $C_0$. The fixed points for this action are
\[
D^{\GL_{n}(\Z_p)} \cong \E^0.
\]
There is a natural transformation of functors defined on finite spaces
\[
D \otimes_{\E^0} \E^0(X) \lra{} C_0 \otimes_{\E^0} \E^0(X) \cong C_{0}^{0}(X)
\]
induced by the inclusion $D \lra{} C_0$.

For $X$ a trivial $G$-space and $\al \in G_{p}^{n}$ the component of the character map corresponding to $[\al]$
\[
\E^0(BG \times X) \lra{} C_{0}^{0}(BC(\im \al) \times X) \cong C_{0}^{0}(X)
\]
factors through the natural transformation above
\[
\E^0(BG \times X) \lra{} D \otimes_{\E^0} \E^0(X).
\]

\subsection{An incomplete list of applications}
The applications of HKR character theory are a bit odd. There is the obvious application to counting the rank of $\E^0(BG)$. This is actually quite effective for a few reasons. First of all, $\E^0(BG)$ is often a free $\E^0$-module (see Proposition 3.5 in \cite{etheorysym}) and so the rank contains a decent amount of information. Secondly, after calculating some ranks one can speculate on the sort of algebro-geometric object that $\E^0(BG)$ might represent by using $\QZ^n$ as a discrete model for the formal group $\G_{\E}$. This is an entertaining form of numerology.

\begin{example}
For instance, take $G = \Sigma_{p^k}$. Then Theorem \ref{main} implies that 
\[
\text{rank}(\E^0(B\Sigma_{p^k})) = |\hom(\Z_{p}^n,\Sigma_{p^k})/\sim|.
\]
A map $\Z_{p}^{n} \lra{} \Sigma_{p^k}$ is precisely a $\Z_{p}^{n}$-set of order $p^k$ and a conjugacy class of such maps corresponds to the isomorphism class of the $\Z_{p}^{n}$-set. Thus the set of conjugacy classes is in bijective correspondence with the set of isomorphism classes of $\Z_{p}^{n}$-sets of order $p^k$. \qed
\end{example}

Ganter has used character theory as a sounding board against which to test the height of certain geometric constructions. For instance, in \cite{Ganter-Kapranov} a theory of $2$-representations is developed and the character map for these representations takes values in class functions on pairs of commuting elements of the group:
\[
Cl(\hom(\Z^2,G),\mathbb{C}).
\]

Beyond this there are several applications in the work of Ando, Rezk, and Strickland and we will endeavor to describe them now.

Ando and Rezk both use character theory to construct cohomology operations for $\E$. In \cite{Isogenies}, Ando constructs higher height analogues of the Adams operations and some operations that he calls Hecke operations. In \cite{logarithmic}, Rezk uses character theory to write down the ``logarithmic element" for $\E$. It is the element in $(\E)_{0}^{\wedge}(\Omega^{\infty}S)$ that controls the logarithmic cohomology operation.

They both go about this by constructing natural transformations
\[
\E^0(X) \lra{} D \otimes_{\E^0} \E^0(X)
\]
that happen to live in the $\GL_n(\Z_p)$-invariants. It's worth saying something about how these transformations are constructed because it is a great application of the Drinfeld ring. Recall that the Goerss-Hopkins-Miller theorem implies that $\E$ is an $E_{\infty}$-ring spectrum. This implies a theory of power operations for $\E$. Thus for any space $X$ and map
\[
X \lra{f} \E
\]
we can apply the $k$th extended powers construction and precompose with the diagonal map to get the composite
\[
\Sigma_{+}^{\infty}B\Sigma_k \times X \lra{\Delta} \Sigma_{+}^{\infty}E\Sigma_k \times_{\Sigma_k} X^k \lra{D_kf} (E\Sigma_k)_{+} \wedge_{\Sigma_k} \E^{\wedge k}.
\]
Now composing with the $E_{\infty}$ structure map
\[
(E\Sigma_k)_{+} \wedge_{\Sigma_k} \E^{\wedge k} \lra{} \E
\]
gives
\[
\Sigma_{+}^{\infty}B\Sigma_k \times X \lra{} \E.
\]
All together this is a (multiplicative, non-additive) map
\[
\E^0(X) \lra{P_k} \E^0(B\Sigma_k \times X),
\]
which is often called the total power operation. A good reference for its properties is Sections 7 and 8 of \cite{logarithmic}. Ando makes the following move: Let $i:A \subset \QZ^n$ be a finite subgroup of order $p^k$ and let $i^*$ be its Pontryagin dual. Now fix a Caley map
\[
A \lra{c} \Sigma_{p^k}.
\]
We can form the following composite:
\[
\E^0(X) \lra{P_{p^k}} \E^0(B\Sigma_{p^k} \times X) \lra{c^*} \E^0(BA \times X) \lra{i^*} \E^0(B\Lk \times X) \lra{} D \otimes_{\E^0} \E^0(X).
\]
Let's call it $\Psi_A$. Note that the composite of the last two maps is a combination of Example \ref{trivG} and the discussion at the end of Subsection \ref{D}. 

It turns out that, due to the relationship between $D$ and level structures for $\G_{\E}$ (that we have hardly discussed), $\Psi_A$ is a ring map. This should be somewhat surprising. 

Given two finite subgroups $A_1, A_2 \subset \QZ^n$ we could send them to the sum 
\[
\Psi_{A_1} + \Psi_{A_2}.
\]
Given an element $e \in \E^0$, we can multiply by it to get $e\Psi_A$. We can also multiply two ring maps using the ring structure in the codomain to get $\Psi_{A_1} \Psi_{A_2}$. 
If we let $\E^0[\sub(\QZ^n)]$ be the polynomial $\E^0$-algebra generated by the finite subgroups $\QZ^n$ then we get a map of $\E^0$-algebras
\[
\E^0[\sub(\QZ^n)] \lra{} \text{Nat}(\E^0(-),D \otimes_{\E^0} \E^0(-)).
\]
What is more, the operation lands in the $\GL_n(\Z_p)$-invariants if it originated from an element in
\[
\E^0[\sub(\QZ^n)]^{\GL_n(\Z_p)}
\]
for the action induced by the action on subgroups of $\QZ^n$. In this case we get a natural transformation
\[
\E^0(-) \lra{} \E^0(-).
\]
For more details of these constructions see Section 3.5 of \cite{Isogenies}. Ando and Rezk construct (somewhat elaborate) elements of $\E[\sub(\QZ^n)]^{\GL_n(\Z_p)}$ that give very interesting cohomology operations.

\begin{example}
One invariant subgroup of $\QZ^n$ is $\Lk = \QZ^n[p^k]$. When $n=1$, this corresponds to a natural ring homomorphism
\[
K_{p}^0(X) \lra{\Psi_{\Lk}} K_{p}^0(X). 
\]
The effect of this map on line bundles is computed in Proposition 3.6.1 of \cite{Isogenies}, although it requires quite a bit of unwrapping, and it sends $L$ to $L^{\otimes p^k}$. Therefore the map is the unstable Adam's operation because the unstable Adam's operation is the unique map with these properties. \qed
\end{example}

In \cite{etheorysym}, Strickland uses another version of the character map as described in Appendix B of \cite{Greenlees-Strickland}. He uses it in the process of giving an algebro-geometric description of the Morava $E$-theory of symmetric groups (modulo a transfer ideal). The result is the following: Let
\[
\sub_{k}(\G_{\E}) : \E^0\text{-algebras} \lra{} \text{Set}
\]
be the functor taking an $\E^0$-algebra $R$ to the set of subgroup schemes of order $p^k$ in $R \otimes \G_{\E}[p^k]$. Also let $I_{tr} \subset \E^0(B\Sigma_{p^k})$ be the ideal generated by the image of the transfer from $\Sigma_{p^{k-1}}^{p} \subseteq \Sigma_{p^k}$. Then there is an isomorphism
\[
\sub_{k}(\G_{\E}) \cong \Spec \E^0(B\Sigma_{p^k})/I_{tr}.
\]
The proof of this theorem requires quite a bit of machinery that is built up in \cite{subgroups} and \cite{etheorysym}. The proof is in Section 9 of \cite{etheorysym}.

The version of character theory used by Strickland in his proof starts with the following algebro-geometric observation:
\begin{prop}
There is an isomorphism
\[
p^{-1} \hom(\Z/p^k,\G_{\E}[p^k]) \cong \Coprod{i \in \{0,\ldots,k\}} p^{-1} \level(\Z/p^i,\G_{\E}[p^k]).
\]
\end{prop}
\begin{proof}
Recall (from Section 7 of \cite{subgroups}, for instance) that the global sections of $\level(\Z/p^i,\G_{\E}[p^k])$ are
\[
\Gamma \level(\Z/p^i,\G_{\E}[p^k]) \cong \E^0\powser{x}/\langle p^i \rangle (x).
\]
Recall that
\[
[p^k](x) = [p^{k-1}](x)\langle p^k \rangle (x) = [p^{k-2}](x)\langle p^{k-1} \rangle (x)\langle p^k \rangle (x) = \Prod{i \in \{0,\ldots,k\}} \langle p^i \rangle (x),
\]
where $\langle p^0 \rangle (x) = x$.

After $p$ has been inverted these factors are pairwise coprime.
Clearly $x$ and $\langle p \rangle (x)$ are coprime as
\[
\langle p \rangle (x) = p + \ldots.
\]
Now note that
\[
\langle p^k \rangle (x) = \langle p \rangle ( [p^{k-1}](x)).
\]
Given $\langle p^{k_1} \rangle (x)$ and $\langle p^{k_2} \rangle (x)$ with $k_1 < k_2$, the series $[p^{k_2-1}](x)$ and $\langle p^{k_2} \rangle (x)$ are coprime using the same relation as the first case above. But $\langle p^{k_1} \rangle (x)$ divides $[p^{k_2-1}](x)$.
\end{proof}

In Appendix B of \cite{Greenlees-Strickland}, Greenlees and Strickland generalize this to produce an isomorphism
\[
p^{-1} \E^0(BG) \lra{\cong} \big (\Prod{A \subseteq G} p^{-1} \Gamma \level(A^*, \G_{\E}[p^k]) \big )^G,
\]
where $A$ runs over all of the abelian subgroups of $G$ and $G$ acts by conjugation on these. This should be thought of as a reformulation of Theorem \ref{main}.

\subsection{All that we can't leave behind}
There are several things that should be mentioned before we finish this introduction. They could be thought of as further directions or just points that ought to be made for completeness. 

The character map can be thought of as a map of Borel equivariant cohomology theories that begins in height $n$ and lands in height $0$. In \cite{tgcm}, the author constructs a version of the character map that is able to land in each height $0 \leq t < n$. The replacement for $C_0$ is a ring $C_t$ that is a $L_{K(t)}\E^0$-algebra, where $K(t)$ is height $t$ Morava $K$-theory. This has an important effect: When $X = \ast$, the codomain of the HKR character map takes the form
\[
\Prod{[\al] \in G_{p}^{n}/\sim}C_{0}^{0}(BC(\im \al)).
\]
Because $C_0$ is a rational algebra we decided that we could largely ignore the centralizers. However, when $C_0$ is replaced by $C_t$ (and $n$ is replaced by $n-t$) these centralizers play an important role. In particular, their cohomology is not trivial. Thus it is harder to view the codomain as ``class functions" on anything.

We have completely ignored the higher cohomology groups in this discussion. Everything can be generalized without too much effort to these (and this is done in both \cite{hkr} and \cite{tgcm}). However, there is an important thing to note: There are finite groups with odd classes in $\E^*(BG)$ (in contrast to $K_p$, see Corollary 7.3 \cite{atiyahcharacters}). This is a theorem of Kriz in \cite{Kriz-odd}. 

Another assumption that we have made through this paper is that we are interested in the classifying spaces of finite groups. This is true, however, a recent theorem of Lurie's extends the character map to all $\pi$-finite spaces (spaces with finitely many nonzero homotopy groups, all of which are finite). The construction of the character map in this case is exactly the same as the topological groupoid definition above, but the proof that it is an isomorphism is more complicated.

\bibliographystyle{abbrv}
\bibliography{mybib}

\end{document}